\documentclass{amsart}
\usepackage{amssymb,amsmath, amsthm,latexsym}
\usepackage{graphics}
\usepackage{psfrag}
\usepackage{amscd}
\usepackage{graphicx}
\newcommand{\cal}[1]{\mathcal{#1}}
\theoremstyle{plain}
\newtheorem{theo}{Theorem}
\newtheorem{cor}{Corollary}

\newtheorem{lemma}{Lemma}[section]
\newtheorem{theorem}[lemma]{Theorem}  
\newtheorem{proposition}[lemma]{Proposition}
\newtheorem{corollary}[lemma]{Corollary}
\theoremstyle{definition}
\newtheorem{defi}{Definition}

\newtheorem{remark}[lemma]{Remark}

\let\egthree=\phi
\let\phi=\varphi
\let\varphi=\egthree

\begin{document}
\title[Asymptotic dimension and disk graphs]
{Asymptotic dimension and the disk graph II}
\author{Ursula Hamenst\"adt}
\thanks{Partially supported by 
ERC Advanced Grant ``Moduli''}
\date{October 6, 2018}


\begin{abstract}
We show that the asymptotic dimension of a hyperbolic 
relatively hyperbolic graph is finite provided that this holds true
uniformly for the 
peripheral subgraphs and for the electrification. 
We use this to show that the asymptotic dimension of the
disk graph of a handlebody of genus $g\geq 2$ is 
at most quadratic in the genus.
\end{abstract}

\maketitle


\section{Introduction}

A metric space
$(X,d)$ has \emph{asymptotic dimension} ${\rm asdim}(X)$
at most $n$  
if for every number $R>0$,
there exists a covering of $X$ by uniformly bounded sets such that
every metric $R$-ball intersects at most $n+1$ of the sets in the 
cover. More generally, 
a collection of metric spaces has asdim at most $n$ \emph{uniformly}
if for every $R$ there are covers of each space whose elements
are uniformly bounded over the whole collection. 


The main goal of this article is to investigate the asymptotic dimension of a 
(not necessarily locally finite) hyperbolic
graph ${\cal G}$ which is hyperbolic relative to a
collection 
$\{H_c\mid c\in {\cal C}\}$ of peripheral subgraphs.
This means the following.
Define the \emph{${\cal H}$-electrification ${\cal E\cal G}$} of ${\cal G}$ 
to be the graph which is obtained from ${\cal G}$ by adding for every
$c\in {\cal C}$ a new vertex $v_c$ which is connected to each 
vertex $x\in H_c$ by an edge and which is not connected
to any other vertex. We require that 
the graph ${\cal E\cal G}$ 
is hyperbolic and that a property called \emph{bounded penetration}
holds true. We shall define this property in Section \ref{asymptotic} 
and refer to \cite{H16} for a detailed discussion. 
We show

\begin{theo}\label{thm1}
Let ${\cal G}$ be a hyperbolic graph which is hyperbolic relative to a family
${\cal H}=\{H_c\mid c\in {\cal C}\}$ of peripheral subgraphs, with electrification
${\cal E\cal G}$. If the collection $H_c$ $(c\in {\cal C})$ has 
${\rm asdim}(H_c)\leq n$ uniformly then 
${\rm asdim}({\cal G})\leq {\rm asdim}({\cal E\cal G})+n+1$.
\end{theo}

Our second goal is to apply Theorem \ref{thm1} to the 
\emph{disk graph} of a \emph{handlebody} of genus $g\geq 2$.
Such a handlebody is 
a compact three-dimensional manifold $H$  which can
be realized as a closed regular neighborhood in $\mathbb{R}^3$
of an embedded bouquet of $g$ circles. Its boundary
$\partial H$ is an oriented surface of genus $g$. 

The disk graph ${\cal D\cal G}$ of $H$ is the metric graph whose 
vertices are isotopy classes of properly embedded disks in $H$ and 
where two such disks are connected by an edge of length one
if they can be realized disjointly. Assigning to a disk its boundary
then defines an embedding of the disk graph into the 
\emph{curve graph} of $\partial H$. However, this inclusion is not
a quasi-isometric embedding \cite{MS13,H16,H17}.

To describe the geometric structure of 
${\cal D\cal G}$ we use the next definition.

\begin{defi}\label{hierarchy2}
A \emph{hierarchy} of a hyperbolic metric graph ${\cal G}$ 
consists of a finite chain ${\cal G}_1,\dots,{\cal G}_k$ 
of hyperbolic graphs with the following properties.
\begin{enumerate}
\item ${\cal G}_k={\cal G}$.
\item  For all $i$, the graph ${\cal G}_{i+1}$ is
hyperbolic relative to a family ${\cal H}_i$ of subgraphs,
with electrification ${\cal G}_i$.
\end{enumerate}
The graph ${\cal G}_1$ is called the \emph{base} of the hierarchy, and
the number $k$ its \emph{depth}. 
\end{defi}

Call the hierarchy \emph{tame} if its base has finite asymptotic
dimension and if moreover for each $i$ there exists
some $n_i$ such that the family ${\cal H}_i$ of subgraphs of 
${\cal G}_{i+1}$ has ${\rm asdim}\leq n_i$ uniformly. 

An inductive application of Theorem \ref{thm1} leads
to 

\begin{cor}\label{finitehier}
The asymptotic dimension of a hyperbolic metric graph ${\cal G}$ 
which admits a tame hierarchy is finite.
\end{cor}

From \cite{H16,H17} we deduce
the following more precise version of the main result of \cite{MS13}. 

\begin{theo}\label{hierarchy} 
The disk graph ${\cal D\cal G}$ of $H$ 
is hyperbolic and admits a tame
hierarchy whose base is a quasi-isometrically embedded subgraph of the 
curve graph of $\partial H$. Furthermore,
\[{\rm asdim}({\cal D\cal G})\leq (3g-3)(2g+2).\]
\end{theo}

Another geometrically defined graph which admits a tame hierarchy
with base a curve graph is the graph of non-separating multicurves
introduced in \cite{H14}. Corollary \ref{finitehier} then
yields that the asymptotic dimension of the graph of 
non-separating multicurves is finite as well.

\bigskip
\noindent 
{\bf Acknowledgement:} This work is based upon work 
supported by the National Science Foundation under Grant No
DMS-1440140 while the author was in residence at the 
Mathematical Science Research Institute in Berkeley,
California, during the Fall 2016 semester.

\section{Asymptotic dimension of hyperbolic relatively 
hyperbolic graphs}\label{asymptotic}

We begin with a general statement about hyperbolic
relatively hyperbolic geodesic metric graphs.
We mostly use the notations from \cite{H16}. 

Consider a connected metric graph ${\cal G}$ in which 
a family ${\cal H}=\{H_c\mid c\in {\cal C}\}$ 
of complete connected subgraphs has been specified. 
Here ${\cal C}$ is a countable, finite or empty index set. 
The graph ${\cal G}$ is \emph{hyperbolic relative to the
family ${\cal H}$} if the following properties are satisfied.

Define the \emph{${\cal H}$-electrification} ${\cal E\cal G}$ of 
${\cal G}$ to be the graph which is obtained from ${\cal G}$
by adding for every $c\in {\cal C}$ a new vertex
$v_c$ which is connected to each vertex in 
$H_c$ by an edge and which is not connected
to any other vertex.
We require that the graph ${\cal E\cal G}$ is hyperbolic 
in the sense of Gromov and that moreover the 
following \emph{bounded penetration property} holds true.

Call a simplicial path $\gamma$ in ${\cal E\cal G}$ \emph{efficient}
if for every $c\in {\cal C}$ we have $\gamma(k)=v_c$ for at most
one $k$. Note that if $\gamma$ is an efficient simplicial
path in ${\cal E\cal G}$ which passes through $\gamma(k)=v_c$ for
some $c\in {\cal C}$ then $\gamma(k-1),\gamma(k+1)\in H_c$.

We require that 
for every $L>1$ there is a number $p(L)>0$ with the following
property. Let $\gamma$ be an efficient $L$-quasi-geodesic in 
${\cal E\cal G}$, let $c\in {\cal C}$ and let 
$k\in \mathbb{Z}$ be such that $\gamma(k)=v_c$. 
If the distance between $\gamma(k-1)$ and $\gamma(k+1)$ 
is at least $p(L)$ then every efficient $L$-quasi-geodesic
$\gamma^\prime$ in ${\cal E\cal G}$ with the same endpoints
as $\gamma$ passes through $v_c$. Moreover, 
if $k^\prime\in \mathbb{Z}$ is such that $\gamma^\prime(k^\prime)=
v_c$ then the distance in $H_c$ between 
$\gamma(k-1),\gamma^\prime(k-1)$ and between
$\gamma(k+1),\gamma^\prime(k+1)$ is at most $p(L)$.
 
A family ${\cal H}=\{H_c\mid c\in {\cal C}\}$ of complete 
connected subgraphs $H_c$ of ${\cal G}$ is 
called \emph{uniformly quasi-convex} if the inclusion $H_c\to {\cal G}$ 
is a quasi-isometric embedding with constant not depending on $c$. 
The following is Theorem 1 of \cite{H16}.

\begin{theorem}\label{hyphyp}
Let ${\cal G}$ be a metric graph which is hyperbolic 
relative to a family ${\cal H}=\{H_c\mid c\in {\cal C}\}$
of complete connected subgraphs. If there is a number 
$\delta >0$ such that each of the graphs $H_c$ is $\delta$-hyperbolic 
then ${\cal G}$ is hyperbolic. Moreover, the 
subgraphs $H_c $ $(c\in {\cal C})$ are 
uniformly quasi-convex.
\end{theorem}

We call a graph ${\cal G}$ with the properties stated in 
Theorem \ref{hyphyp} 
a \emph{hyperbolic relatively hyperbolic graph}. In the sequel 
we always assume that all assumptions in Theorem \ref{hyphyp} are
fulfilled. 

We begin with collecting some easy geometric properties of 
a hyperbolic relatively 
hyperbolic graph as defined above.
To this end recall that for every quasi-convex
subgraph $H$ of a hyperbolic graph ${\cal G}$ there is a coarsely
well defined shortest distance projection $\Pi_H:{\cal G}\to H$, i.e. 
a projection which associates to a point in ${\cal G}$ a choice of 
a point in $H$ of approximate shortest distance. Any other choice of such a
point is of    
bounded distance, and this bound only depends on the
hyperbolicity constant of ${\cal G}$ and the constant
defining the quality of the quasi-isometric embedding
$H\to {\cal G}$. The map $\Pi_H$  is 
coarsely distance non-increasing, i.e. it increases
distances at most by a fixed additive constant.

The following lemma shows that the subgraphs $H_c$ of ${\cal G}$ fulfill
the three
axioms in Theorem B of \cite{BBF15}.

\begin{lemma}\label{projectionbound}
Let ${\cal G}$ be a hyperbolic graph which is hyperbolic relative to 
a family ${\cal H}=\{H_c\mid c\in {\cal C}\}$ 
of complete connected subgraphs. We require that 
these subgraphs are $\delta$-hyperbolic for some fixed
number $\delta >0$. 
\begin{enumerate}
\item There is a number $R>0$ such that for $c\not=d\in {\cal C}$, 
${\rm diam}(\Pi_{H_c}(H_d))\leq R$.
\item
For $a,b,c\in {\cal C}$ define 
\[d_a(b,c)={\rm diam}(\Pi_{H_a}(H_b)\cup \Pi_{H_a}(H_c)).\]
There exists a constant $\theta>0$ such that for any triple of distinct
elements $a,b,c\in {\cal C}$, at most one of the numbers
\[d_a(b,c),d_b(a,c),d_c(a,b)\] is greater than $\theta$.
\item For any $a,b\in {\cal C}$, the set 
\[\{c\in {\cal C}\mid d_c(a,b)>\theta\}\]
is finite. 
\end{enumerate}
\end{lemma} 
\begin{proof} To show the first property, 
recall from Theorem \ref{hyphyp} that 
the subgraphs $H_c$ of ${\cal G}$ 
are uniformly quasiconvex. 
By hyperbolicity of ${\cal G}$, this 
implies that there exists a
number $D>0$ such that 
for $c\in {\cal C}$, any geodesic in ${\cal G}$ connecting
two points $x,y\in H_{c}$
is contained in the $D$-neighborhood $N_D(H_c)$ of $H_c$.

For $c\in {\cal C}$ write 
\[\Pi_c=\Pi_{H_c}:{\cal G}\to H_c.\]
Using again hyperbolicity, we deduce the following. Let us assume that
the diameter of the projection $\Pi_c(H_d)$ is large. 
Then there exist points $x,y\in H_d$ and a geodesic 
$\zeta:[a,b]\to {\cal G}$ connecting $\zeta(a)=x$ to 
$\zeta(b)=y$ which is contained both in a uniformly bounded
neighborhood of $H_c$ as well as
in a uniformly bounded neighborhood of $H_d$. 
However, this violates the bounded
penetration property. Namely, 
we can find an efficient quasi-geodesic $\gamma$ in 
${\cal E\cal G}$ connecting $x$ to $y$ which does not pass through
$v_d$ (but instead passes through $v_c$). 
We refer to the work
\cite{Si12} for a more detailed discussion of the various 
equivalent formulations of relative hyperbolicity, in particular in 
connection to the condition $(\alpha_1)$ formulated in \cite{Si12} for the collection 
${\cal H}=\{H_c\mid c \in {\cal C}\}$.

To show the second property, let $R>0$ be such that the first property
is valid for this $R$. By hyperbolicity of ${\cal G}$ and uniform
quasi-convexity of 
the subspaces $H_c\subset {\cal G}$ $(c\in {\cal C})$, 
there exists a number 
$R^\prime >R$ such that for all  
$x,y\in {\cal G}$ with  
$d(\Pi_c(x),\Pi_c(y))\geq R^\prime$ the following two
properties are satisfied.
\begin{enumerate}
\item[(a)] Any geodesic 
in ${\cal G}$ connecting $x$ to $y$ passes through a
uniformly bounded neighborhood of both $\Pi_c(x)$ and
$\Pi_c(y)$, say through the $R^{\prime\prime}$-neighborhood.
\item[(b)] Any geodesic in ${\cal E\cal G}$ connecting 
$x$ to $y$ passes through the special vertex $v_c$.
\end{enumerate}

Now assume that 
$d_a(b,c)\geq 3R^\prime$. 
Choose a point $x\in \Pi_b(H_a)$ and let 
$y=\Pi_c(\Pi_a(x))\in \Pi_c(H_a)$; then $d(\Pi_a(x),\Pi_a(y))>R^\prime$. 
By the choice of $R^\prime>R$ and by property (1) in the lemma, 
any geodesic connecting $x$ to $y$ passes
through the $R^{\prime\prime}$-neighborhood of 
$\Pi_a(x)$. By hyperbolicity and uniform quasi-convexity of the 
subgraphs $H_c$, 
the projection $\Pi_c(x)$ is uniformly 
near the projection $\Pi_c(z)$ of any point $z$ on a geodesic 
connecting $x$ to a point on $H_c$ provided that the distance
between $z$ and $H_c$ is sufficiently large. Since furthermore
the projection $\Pi_c$ is coarsely distance non-increasing,  
the distance between $\Pi_c(x)$ and $\Pi_c(\Pi_a(x))$ is 
uniformly bounded. By the first part of 
the lemma, the same then holds true for 
$d_c(a,b)$. The same reasoning also shows that
$d_b(a,c)$ is uniformly bounded. To summarize,
the second statement in the lemma 
holds true for some 
number $\theta\geq 3R^\prime$.

We are left with showing the third property. 
To this end let 
$x\in \Pi_a(H_b),y\in \Pi_b(H_a)$ and let $\gamma$ be a geodesic
connecting $x$ to $y$ in ${\cal E\cal G}$. Then
$\gamma$ passes through only finitely many of the
special vertices $v_c$, say through the vertices 
$v_{c_1},\dots,v_{c_u}$.
However, 
by the above choice of the number $R^\prime>0$, 
if $c\in {\cal C}$ is such that $d_c(a,b)>\theta\geq 3R^\prime$ then 
by property (b) above, any geodesic in 
${\cal E\cal G}$ with endpoints $x,y$ passes through the 
special vertex $v_c$. 
This shows that if 
$d_c(a,b)\geq \theta$ then $c\in \{c_1,\dots,c_u\}$. 
This completes the proof of the lemma.
\end{proof}

We use Lemma \ref{projectionbound}  
to show Theorem \ref{thm1} from the introduction.

\begin{theorem}\label{provethm1}
Let ${\cal G}$ be a hyperbolic metric graph which is hyperbolic 
relative to a family ${\cal H}=\{H_c\mid c\in {\cal C}\}$
of complete connected uniformly hyperbolic 
subgraphs, with ${\cal H}$-electrification
${\cal E\cal G}$. If ${\rm asdim}(H_c)\leq n$ uniformly then 
${\rm asdim}({\cal G})\leq {\rm asdim}({\cal E\cal G})+n+1$.
\end{theorem}
\begin{proof} By Theorem B of \cite{BBF15} and Lemma \ref{projectionbound},
there exists a quasi-tree of metric spaces
${\cal Y}$ which is built from 
the subgraphs $H_c$ $(c\in {\cal C})$. The space 
${\cal Y}$ is a connected geodesic metric graph, and it contains
each of the graphs $H_c$ as a subgraph. The vertices of ${\cal Y}$ 
are the vertices of the graphs $H_c$ $(c\in {\cal C})$. 
Our goal is to construct a quasi-isometric embedding of ${\cal G}$ into
the product of ${\cal E\cal G}$ with 
${\cal Y}$. 

This is sufficient for the purpose of the theorem. Namely, 
using the assumption on the asymptotic
dimensions of the graphs $H_c$, 
Theorem B iv) of \cite{BBF15} states that ${\rm asdim}({\cal Y})\leq n+1$.
Now the asymptotic 
dimension of the product $X\times Y$ of two metric spaces
satisfies ${\rm asdim}(X\times Y)\leq {\rm asdim}(X)+{\rm asdim}(Y)$,
furthermore ${\rm asdim}(X)\leq {\rm asdim}(Y)$ if 
$X$ admits a 
quasi-isometric embedding into $Y$ (see \cite{BD06}).
Thus if ${\cal G}$ admits a quasi-isometric embedding into
${\cal E\cal G}\times {\cal Y}$ then 
\[{\rm asdim}({\cal G})\leq {\rm asdim}({\cal E\cal G}\times {\cal Y})\leq 
{\rm asdim}({\cal E\cal G})+n+1.\]

The graph \cite{BBF15} 
${\cal Y}$ is the union of the graphs 
$H_c$ $(c\in {\cal C})$ and a collection of additional edges of length one
connecting these graphs . 
These edges are chosen as follows.

For all $c,d$, let $x_{c,d}\in \Pi_{c}(H_d)\subset 
H_c$ be a point of shortest distance 
in $H_c$ to the graph $H_d$.  
Connect the point $x_{c,d}$ to the point
$x_{d,c}$ by an edge if there does not exist an 
$a\in {\cal C}$ so that 
$d_a(c,d)\geq 2\theta$ for the threshold $\theta>0$
from part (2) of Lemma \ref{projectionbound}.  

Let $\gamma$ be a geodesic in ${\cal E\cal G}$
and let $v_{c_1},\dots,v_{c_s}$ $(c_i\in {\cal C})$ 
be the special points on $\gamma$.
Define
$v_{c_i}$ to be \emph{wide} for $\gamma$ if the following holds true.
Let $k_i>0$ be such that $\gamma(k_i)=v_{c_i}$; then the distance in 
$H_{c_i}$ between $\gamma(k_i-1)$ and $\gamma(k_i+1)$ is at least 
$\theta$. With this terminology, 
the points $x_{c,d}$ and $x_{d,c}$ are connected
by an edge of length one if there exists a geodesic $\gamma$ 
in ${\cal E\cal G}$ connecting $x_{c,d}$ to $x_{d,c}$ 
which does not contain any wide points.

We now define a map ${\cal G}\to {\cal E\cal G}\times {\cal Y}$ as follows.
Fix a basepoint $x\in {\cal G}$ contained in one of the 
quasi-convex subspaces 
$H_c$. Associate to $x$ the product $(x,x)\in {\cal E\cal G}\times {\cal Y}$.
For every vertex $y\in {\cal G}$ choose once and for all a geodesic 
$\gamma_y$ in ${\cal E\cal G}$ 
connecting $x$ to $y$.
Note that such a geodesic is efficient.

Let $v_{c_1},\dots,v_{c_s}\in {\cal E\cal G}$ 
be the special points traveled through by $\gamma_y$
in this order $(c_i\in {\cal C})$. 
Let $k_s>0$ be such that 
$v_{c_s}=\gamma_y(k_s)$ and define
\[\Psi(y)=\gamma(k_s+1)\in H_{c_s}\subset {\cal Y}.\]

We claim that the map 
\[\Lambda:y\to \Lambda(y)=(y,\Psi(y))\in 
{\cal E\cal G}\times {\cal Y}\] 
is a quasi-isometric embedding.
To this end we show first that for any two vertices
$y,z\in {\cal G}$ of distance one, 
the distance between $\Lambda(y)$ and $\Lambda(z)$ is uniformly bounded.
Since ${\cal G}$ is a geodesic metric space, this then implies that
the map $\Lambda$ is coarsely Lipschitz.

Consider the geodesics $\gamma_y,\gamma_z$ in ${\cal E\cal G}$.
Let $v_{c_1},\dots,v_{c_s}$  be the
special points on $\gamma_y$ and let 
$w_{d_1},\dots,w_{d_u}$ be those on $\gamma_z$
$(c_i,d_j\in {\cal C})$.
Assume that $i\leq s$ is the largest number so that
$v_{c_i}\in \gamma_z$.
By the bounded penetration property and the choice of 
$\theta$, for no $j>i$ the vertex $v_{c_j}$ is 
wide for $\gamma_y$. 
If $\ell_j$ is such that $\gamma_z(\ell_j)=w_{d_j}=v_{c_i}$  
(in fact, as $\gamma_y,\gamma_z$ are both geodesics with the
same initial point, we must have $\ell_j=k_i$, however
this fact is not important for us), then using once more the
bounded penetration property, 
the distances between the exit points $q=\gamma_y(k_i+1)$
of $\gamma_y$ and $\gamma_z(\ell_j+1)$ of $\gamma_z$ 
from $H_{c_i}=H_{d_j}$ is uniformly bounded.

It now suffices to show that the distance in ${\cal Y}$ between 
$\Psi(y)$ and the exit point $q$ of $\gamma_y$ in $H_{c_i}$ 
is uniformly bounded.
However, by definition, $\Psi(y)$ is the exit point of the intersection of 
$\gamma_y$ with 
$H_{c_s}$. Thus if
$i=s$ then $q=\Psi(y)$ and we are done. 
Otherwise note that by uniform quasi-convexity, the  
exit point $q$ of $\gamma_y$ from $H_{c_i}$ 
is contained in a uniformly bounded
neighborhood of $\Pi_{{c_i}}(H_{c_s})$. 
Thus by the first part of 
Lemma \ref{projectionbound} and up to adjusting constants,
the shortest distance projection
of $H_{c_s}$ into $H_{c_i}$ is contained in the 
$3R^\prime<\theta$-neighborhood of $q$ (compare also \cite{BBF15}).

Furthermore, as none of the vertices $v_{c_{i+1}},\dots,v_{c_s}$ along the 
segment of $\gamma_y$ 
connecting $q$ to $\Psi(y)\in H_{c_s}$ is wide,
there is a segment in ${\cal Y}$ of length one
connecting a point in $H_{c_i}$ uniformly near $q$ 
to a point in $H_{c_s}$ uniformly
near the entry point of $\gamma_y$ in $H_{c_s}$. As $v_{c_s}$ is not 
wide along $\gamma_y$,
this entry point is contained in a uniformly bounded neighborhood of 
$\Psi(y)$. Together this shows that indeed, 
the distance in ${\cal Y}$ between $q$ and $\Psi(y)$ is uniformly bounded.
The same reasoning applies to $\gamma_z$ and shows that the distance in 
${\cal Y}$ between $\Psi(z)$ and $\gamma_z(\ell_j+1)\in H_{c_i}$ 
is uniformly bounded. Since the distance in $H_{c_i}$ between
$\gamma_y(k_i+1)$ and $\gamma_z(\ell_j+1)$ is uniformly bounded, 
the distance between $\Psi(y),\Psi(z)$ is uniformly bounded.

We showed so far that the map $\Lambda$ is coarsely Lipschitz, and 
we are left with showing that there exists a constant $L>1$ such that 
for all $y,z\in {\cal G}$, the distance
between $\Lambda(y)$ and $\Lambda(z)$ is bounded from below by
$d(y,z)/L-L$. 

Following \cite{H16}, define the \emph{enlargement} $\hat\gamma$ of a
geodesic $\gamma:[0,n]\to {\cal E\cal G}$ with endpoints 
$\gamma(0),\gamma(n)\in {\cal G}$ as follows. Let 
$0< k_1<\cdots <k_s<n$ be those points such that 
$\gamma(k_i)=v_{c_i}$ for some $c_i\in {\cal C}$. Then 
$\gamma(k_i-1),\gamma(k_i+1)\in H_{c_i}$. For each 
$i\leq s$ replace $\gamma[k_i-1,k_i+1]$ by a simplicial geodesic
in $H_{c_i}$ with the same endpoints. 

 Theorem 2.4 of \cite{H16} shows that enlargements of geodesics in 
 ${\cal E\cal G}$ are uniform quasi-geodesics in ${\cal G}$. 
 Thus it suffices to show the existence of a number $L>1$ with
the following property. If 
$\hat \gamma$ is the enlargement of any geodesic in ${\cal E\cal G}$,
 parametrized as a simplical edge path, then 
 \[d(\Lambda(\hat \gamma(m)),\Lambda(\hat  \gamma(n)))\geq \vert n-m\vert/L-L\]
 for all $m,n$.

To this end we first show the following.
Let $y\in {\cal G}\subset {\cal E\cal G}$ and let 
$\hat \gamma:[0,R]\to {\cal G}$ be 
an enlargement of the geodesic $\gamma_y$ connecting the basepoint 
$x=\hat \gamma(0)$ to $y=\hat \gamma(R)$.
Then for $0\leq u<R$ we have  
$d(\Lambda(\gamma(u)),\Lambda(\gamma(R)))\geq \vert R-u\vert/L-L$.

Let $s\geq u$ be the maximum of all numbers so that
$\hat \gamma[u,s]\in H_c$ for some $c\in {\cal C}$. If $\gamma(u)$ is not contained 
in any of the special subspaces then put $s=u$. By construction of an enlargement,
we have $\hat \gamma(s)=\gamma_y(t_0+1)$ for some $t_0$. 
Let $v_{c_1},\dots,v_{c_j}$ $(c_i\in {\cal C})$ 
be the special 
points passed through by
the geodesic $\eta=\gamma_y [s,R]$.
Let us suppose that $v_{c_i}=\gamma_y(t_i)$ for some $t_i$; then
$\gamma_y(t_i-1),\gamma_y(t_i+1)\in H_{c_i}\subset {\cal Y}$. 
With a small abuse of notation, 
write also $\gamma_y(t_0-1)=\hat \gamma(u)$.

Denote by $d_{H_{c_i}}$ the intrinsic path metric in $H_{c_i}$.
By construction, the length $R-u$ of $\hat\gamma[u,R]$
is not larger than  
\[c \bigl(d_{\cal E}(\hat\gamma(u),\hat\gamma(R))+\sum_{i\geq 0}
d_{H_{c_i}}(\gamma_y(t_i-1),\gamma_y(t_i+1))\bigr)\]
for a universal constant $c>0$ which depends on the choices of the
constants in the construction of the space ${\cal Y}$
and arises from ignoring arcs of uniformly bounded length in 
special subspaces $H_c$ obtained by enlarging 
subsegments of $\gamma_y$ through special points which is not wide
(recall that the length of such a segment is uniformly bounded). 
On the other hand, by Theorem 4.13 of \cite{BBF15}, 
the distance in the graph ${\cal Y}$ between 
$\Psi(\gamma_y(t_0-1))=\hat \gamma(u)$ and $\Psi(\gamma_y(t_j+1))$ 
is proportional to $\sum_i d_{H_{c_i}}(\gamma_y(t_i-1),
\gamma_y(t_i+1))$. Together we conclude that indeed, $\Lambda$ coarsely
decreases the distance along geodesic segments from the initial point $x$
by at most a uniform multiplicative constant.

Now let $y,z\in {\cal G}$ be arbitrary vertices and let
$\hat \gamma$ be the enlargement of a  geodesic  $\gamma$ in ${\cal E\cal G}$ 
connecting $y$ to $z$. 
Let $\hat \gamma_y,\hat \gamma_z$ be enlargements
of the geodesics $\gamma_y,\gamma_z$ in ${\cal E\cal G}$ connecting
the fixed basepoint $x$ to $y,z$. Assume that
$\hat \gamma_y$ is parametrized on $[0,T_1]$ and 
$\hat\gamma_z$ is parametrized on $[0,T_2]$.
By hyperbolicity of ${\cal G}$, there exists a point on 
$\hat \gamma$, say the point $\hat \gamma(u)$, which is uniformly 
near a point $\hat\gamma_y(s)$ 
on $\hat \gamma_y$ and a point $\hat \gamma_z(t)$ 
on $\hat \gamma_z$. 
As $\Lambda$ is coarsely Lipschitz,  the images of these three points under
the map $\Lambda$ are uniformly close.

We distinguish now two cases. In the first case, there are numbers
$s^\prime\geq s,t^\prime\geq t$ so that $\Psi(y)=\hat \gamma_y(s^\prime),
\Psi(z)=\hat \gamma_z(t^\prime)$. 
Using Theorem 4.13 of \cite{BBF15}, the distance between 
$\Psi(y)$ and $\Psi(z)$ is proportional 
to the sum of the distances in the subgraphs $H_{c_i}$ between
entry and exit point of $\gamma$ where 
the $c_i\in {\cal C}$ are those points for
which $v_{c_i}$ is wide along $\gamma$. As in the previous
paragraph, this yields the estimate we are looking for. 

If say $s^\prime\leq s$, then the intersection of $\hat\gamma_y[s,T_1]$ with 
any of the subgraphs $H_c$ has uniformly bounded diameter. 
Futhermore, the wide points along $\gamma$ coincide with a subset of the
wide points along $\gamma_z$. 
Using once more Theorem 4.13 of \cite{BBF15}, we conclude as before that
the distance between $\Psi(y)$ and $\Psi(z)$ is proportional
to the sum of the distances in the subgraphs $H_{c_i}$ between
entry and exit point of $\hat \gamma$ where 
the set $c_i\in {\cal C}$ consists of those points for
which $v_{c_i}$ is wide along $\gamma$. 
Together this completes the proof of the theorem.
 \end{proof}

\begin{remark}\label{uniform}
It is immediate from the proof of Theorem \ref{provethm1} that 
the bounds on coverings used in the definition of asymptotic dimension
for a hyperbolic relatively hyperbolic graph ${\cal G}$ 
are uniformly controlled by
the data of the peripheral subgraphs and the bounds for the 
electrification ${\cal E\cal G}$ of ${\cal G}$. 
\end{remark}

\begin{remark}\label{quasitree}
The proof of Theorem \ref{provethm1} together with
Theorem B of \cite{BBF15} also shows the following.
If ${\cal E\cal G}$ is a quasi-tree and if each of the peripheral
subgraphs are quasi-trees, then ${\cal G}$ admits an
quasi-isometric embedding into the product of two quasi-trees. 
\end{remark}

\section{A tame hierarchy for the disk graph}\label{disk}

The goal of this section is to apply Theorem \ref{thm1} to two geometric
graphs which are related to surfaces.

Let us consider a handlebody $H$ of genus $g\geq 2$. This is a compact
3-manifold with boundary $\partial H$ which is a regular neighborhood of 
a bouquet of $g$ circles in $\mathbb{R}^3$. A \emph{disk} in $H$ is 
a properly embedded disk $D\subset H$ whose boundary is a non-contractible
curve in $\partial H$.

A connected essential subsurface $X$ of $\partial H$ is called \emph{thick}
if the following holds true.
\begin{enumerate}
\item 
Every disk in $H$ intersects $X$.
\item If $c\subset X$ is a simple closed curve which is disjoint
from all boundaries of disks which are completely contained 
in $X$ then $c$ is either contractible or homotopic into the boundary
of $X$.
\end{enumerate}
An example of a thick
subsurface is the entire boundary $\partial H$.

The \emph{disk graph} ${\cal D\cal G}(X)$ of $X$ is the 
graph whose vertices are isotopy classes of disks with boundary in $X$
and where two such disks
are connected by an edge of length one if 
they can be realized disjointly. 

In \cite{H16,H17}, we defined two more graphs whose vertices are
isotopy classes of disks with boundary in $X$.
The \emph{electrified disk graph}
${\cal E\cal D\cal G}(X)$ is obtained from ${\cal D\cal G}(X)$
by adding an edge between any two disks with boundary in $X$
which are disjoint from a common 
essential simple closed curve in $X$.

An \emph{$I$-bundle generator} in a thick subsurface $X$ is an
essential simple closed curve $\gamma\subset \partial H$ with the 
following property.
There exists a compact surface $F$ with non-empty boundary $\partial F$, 
and there is an orientation preserving embedding $\Psi$ of the oriented 
$I$-bundle ${\cal J}(F)$ over $F$ into $H$ which maps a boundary component
$\alpha$ of $\partial F$ to $\gamma$ and which maps the union of
the $I$-bundle over $\alpha$ with the 
\emph{horizontal 
boundary} of ${\cal J}(F)$ (i.e. the subset of the boundary which is
disjoint from the interiors of the intervals of the $I$-bundle) 
to the complement in 
$X$ of a tubular neighborhood of the boundary $\partial X$ of $X$.

The \emph{superconducting disk graph} ${\cal S\cal D\cal G}(X)$ 
is obtained from the electrified disk graph ${\cal E\cal D\cal G}(X)$
by adding an edge of length one between any two disks whose boundaries
intersect an $I$-bundle generator $\gamma$ in $X$ in precisely two points. 

Denote by ${\cal C\cal G}(X)$ the \emph{curve graph} of $X$.
This is the hyperbolic geodesic metric graph whose
vertices are essential simple closed curves in $X$ and where
two such vertices are connected by an edge of length one if and
only if they can be realized disjointly. 
The following is Theorem 5.2 of \cite{H17}.

\begin{theorem}\label{quasi}
There is an effectively computable number $L>1$ only depending on the 
genus of $H$ 
such that for every
thick subsurface $X$ of $\partial H$, 
the vertex inclusion which maps a disk with boundary in $X$ to its
boundary defines an $L$- quasi-isometric embedding
${\cal S\cal D\cal G}(X)\to {\cal C\cal G}(X)$. 
\end{theorem}

Denoting by $\chi(X)$ the Euler characteristic of $X$, we obtain 
as an immediate consequence
of Theorem \ref{quasi} the following 

\begin{corollary}\label{asymptotic1}
The graphs ${\cal S\cal D\cal G}(X)$ are hyperbolic, and of 
asymptotic dimension at most $2\vert \chi(X)\vert$, 
uniformly.
\end{corollary}
\begin{proof} A quasi-isometrically embedded geodesic
metric subgraph of 
a hyperbolic geodesic metric graph satisfies the thin triangle condition and 
hence it is hyperbolic. 
 
It is immediate from the definition 
that the existence of a quasi-isometric embedding
$f:Y\to Z$ implies that ${\rm asdim}(Y)\leq {\rm asdim}(Z)$. 
Now the asymptotic dimension of the curve graph of 
an oriented surface $X$ of finite
type is at most $2\vert \chi(X)\vert$ 
\cite{BB15} uniformly (see also the earlier work 
\cite{BF08} for finiteness) and therefore by Theorem \ref{quasi}, the asymptotic
dimension of ${\cal S\cal D\cal G}(X)$ is at 
most $2\vert \chi(X)\vert$ uniformly as claimed.
\end{proof}

Let again $X$ be a thick subsurface of $\partial H$ and let $\gamma$ be an
$I$-bundle generator in $X$. Denote by ${\cal E}(\gamma)$ the subgraph of 
${\cal D\cal G}(X)$ of all disks 
which intersect $\gamma$ in precisely two points. 
We have

\begin{lemma}\label{asympuni}
The graphs ${\cal E}(\gamma)$ satisfy 
 ${\rm asdim}\leq \vert \chi(X)\vert+2$
uniformly. 
\end{lemma}
\begin{proof}
By Lemma 4.2 of \cite{H16}, the map which 
associates to a disk $D\in {\cal E}(\gamma)$
the projection of $\partial D$ to the base surface $F$ of the 
$I$-bundle corresponding to $\gamma$  
extends to a 2-quasi-isometry of ${\cal E}(\gamma)$ onto the 
\emph{electrified arc graph} of $F$. This electrified 
arc graph is $4$-quasi-isometric to the curve graph of $F$
(see Lemma 4.1 of \cite{H16} for a proof of this folklore result). 
Thus by the main result in \cite{BB15}, 
the asymptotic dimension of the graph ${\cal E}(\gamma)$
is bounded from above by $4g(F)-3+p=-2\chi(F)+1-p$
uniformly,
where $g(F)$ is the
genus of $F$ and where $p\geq 1$ is the number of boundary components. 
The Lemma now follows from the fact that 
 $X$ is obtained from a two-sheeted cover of $F$ by attaching an 
annulus to two boundary components and  that furthermore 
the Euler characteristic of $X$ is negative (see \cite{H16}).
\end{proof}

The following summarizes Lemma 4.2, Corollary 4.3, Lemma 4.5 and Corollary 4.6
of \cite{H16}.

\begin{theorem}\label{electrified}
The electrified disk
graph ${\cal E\cal D\cal G}(X)$ of a thick subsurface $X$ of $\partial H$
is hyperbolic relative to the collection
of subgraphs ${\cal E}(\gamma)$ where $\gamma$ runs through all $I$-bundle
generators of $X$, with electrification
${\cal S\cal D\cal G}(X)$.
\end{theorem}

Using Theorem \ref{electrified}, Lemma \ref{asympuni} and 
Corollary \ref{asymptotic1} we can now show

\begin{corollary}\label{thickasymp}
Let $X\subset \partial H$ be a thick subsurface of genus $g(X)\geq 0$ with 
$p$ boundary components; then ${\rm asdim}({\cal E\cal G}(X))\leq 
\vert \chi(X)\vert +3$.
\end{corollary}
\begin{proof} By Theorem \ref{electrified}, 
the graph ${\cal E\cal G}(X)$ is hyperbolic and hyperbolic relative to the 
subgraphs ${\cal E}(\gamma)$, with electrification the graph 
${\cal S\cal D\cal G}(X)$. By Lemma \ref{asympuni}, the asymptotic
dimension of each of the graphs ${\cal E}(\gamma)$ does not exceed
$\vert \chi(X)\vert +2$., and Corollary \ref{asymptotic1} shows that the
asymptotic dimension of ${\cal S\cal D\cal G}(X)$ is not bigger than
$2\vert \chi(X)\vert$. Thus 
Proposition \ref{provethm1} implies that
${\rm asdim}({\cal E\cal G}(X))\leq \vert \chi(X)\vert +3$ as claimed.
\end{proof}

As a consequence of Corollary \ref{thickasymp} and the results in \cite{H16} we
are now ready to show

\begin{theorem}\label{as}
The asymptotic dimension of the disk graph of a handlebody of genus $g$ is 
bounded from above by $(3g-3)(2g+2)$.
\end{theorem}
\begin{proof}
By the main result of \cite{H16} and the above estimates of asymptotic dimension, 
there is a sequence ${\cal G}_1,\dots,{\cal G}_{3g-3}$ of hyperbolic graphs with the 
following properties.
\begin{enumerate}
\item ${\cal G}_1={\cal E\cal G}(\partial H)$.
\item For each $i\geq 2$, ${\cal G}_i$ is a hyperbolic graph which is hyperbolic
relative to a family ${\cal H}$ of hyperbolic subgraphs. The 
${\cal H}$-electrification of ${\cal G}_i$ equals the graph ${\cal G}_{i-1}$. 
The asymptotic dimension of each graph $H$ in the family is at most
$2g+1$. 
\end{enumerate}

The graph ${\cal G}_i$ is defined as follows. Its vertices are disks, and two 
vertices are connected by an edge of length one if either they are disjoint or
if they are disjoint from an essential multicurve in $X$ with at least $i$ components.
It is shows in \cite{H16} that for $i\geq 2$, 
the graph ${\cal G}_i$ is hyperbolic relative to a family of complete connected 
subgraphs ${\cal E\cal D\cal G}(i-1)$ where such a subgraph is 
the electrified disk graph of a thick subsurface $X$ which is the complement in 
$\partial H$ of a multicurve with $i-1$ components.  

Applying Proposition \ref{provethm1} inductively $3g-3$ times 
and using Corollary \ref{thickasymp}, we conclude that 
the asymptotic dimension of the disk graph is at most $(3g-3)(2g+2)$.
\end{proof}

For a surface $S$ of finite type of genus 
$g\geq 2$ with $m\geq 0$ punctures define the 
\emph{graph of non-separating curves} ${\cal N\cal C}$ to be
the complete subgraph of the curve graph of $S$ consisting of 
non-separating curves. 

We showed in \cite{H14} that this graph is 
hyperbolic. Furthermore, it admits a tame hierarchy
with base a quasi-isometrically embedded subgraph of a curve graph
and where each peripheral graph also is a quasi-isometrically embedded
subgraph of some curve graph.  Thus we obtain

\begin{proposition}\label{nonseparating}
For any surface $S$ of finite type, 
the graph of non-separating  curves on $S$ has finite asymptotic 
dimension.
\end{proposition}

\bigskip\bigskip

\noindent
MATH. INSTITUT DER UNIVERSIT\"AT BONN\\
ENDENICHER ALLEE 60\\
53115 BONN\\
GERMANY\\

\bigskip\noindent
e-mail: ursula@math.uni-bonn.de


\begin{thebibliography}{HPW13}


\bibitem[BD06]{BD06}
G.~Bell, A.~Dranishnikov, {\em Asymptotic dimension},
Topology Appl. 155 (2008), 1265--1296.

\bibitem[BF08]{BF08} G.~Bell, K.~Fujiwara,
{\em The asymptotic dimension of the curve
graph is finite}, J. Lond. Math. Soc. 77 (2008),
33--50.




\bibitem[BBF15]{BBF15} M.~Bestvina, K.~Bromberg,
K.~Fujiwara, {\em Constructing group actions 
on quasi-trees and applications to mapping class groups}.
Publ. Math. Inst. Hautes Etudes Sci. 122 (2015), 1--64.


\bibitem[BB15]{BB15} M.~Bestvina, K.~Bromberg,
{\em On the asymptotic dimension of the curve complex},
arXiv:1508.04832.
















\bibitem[H14]{H14} U.~Hamenst\"adt, {\em Hyperbolicity
of the graph of non-separating multicurves}, 
Alg. \& Geom. Topol. 14 (2014), 1759--1778. 

\bibitem[H16]{H16}  U.~Hamenst\"adt, 
{\em Hyperbolic relatively hyperbolic graphs and disk graphs},
Groups, Geom. Dyn. 10 (2016), 1--41.

\bibitem[H17]{H17} U.~Hamenst\"adt, 
{\em Asymptotic dimension and the disk graph I}, 
preprint 2011. 



































 








\bibitem[MS13]{MS13} H.~Masur, S.~Schleimer, 
{\em The geometry of the disk complex}, 
J. Amer. Math. Soc. 26 (2013), 1--62.















\bibitem[Si12]{Si12} A. Sisto, {\em On metric relative
hyperbolicity}, arXiv:1210.8081.


\end{thebibliography}
\end{document}